\documentclass[psamsfonts,onesided,12pt]{amsart}
\usepackage{amsmath,amssymb, amsthm, graphicx, hyperref}
\usepackage[all,arc,cmtip]{xy}
\usepackage{enumerate}

\usepackage{tikz}
\usetikzlibrary{arrows,chains,matrix,positioning,scopes}
\makeatletter
\tikzset{join/.code=\tikzset{after node path={%
\ifx\tikzchainprevious\pgfutil@empty\else(\tikzchainprevious)%
edge[every join]#1(\tikzchaincurrent)\fi}}}

\makeatother
\tikzset{>=stealth',every on chain/.append style={join}, every join/.style={->}}

\usepackage{mathrsfs}
\usepackage[margin=1in]{geometry}

\numberwithin{equation}{section}

\theoremstyle{plain}
\newtheorem{theorem}[equation]{Theorem}
\newtheorem{lemma}[equation]{Lemma}
\newtheorem{proposition}[equation]{Proposition}
\newtheorem{corollary}[equation]{Corollary}

\theoremstyle{definition}
\newtheorem{definition}[equation]{Definition}

\newtheorem*{convention}{Convention}

\theoremstyle{remark}

\newcommand{\mc}[1]{\mathcal{#1}}
\newcommand{\mf}[1]{\mathfrak{#1}}

\newcommand{\op}[1]{\operatorname{#1}}

\newcommand{\st}{\text{ } \big| \text{ }}

\newcommand{\inv}{^{-1}}

\newcommand{\Z}{\mathbb{Z}} \newcommand{\N}{\mathbb{N}}

\renewcommand{\H}{\operatorname{H}}

\newcommand{\Hom}{\operatorname{Hom}}
\newcommand{\End}{\operatorname{End}}
\newcommand{\Ext}{\operatorname{Ext}}

\newcommand{\coker}{\operatorname{coker}}

\newcommand{\define}[1]{\emph{#1}}
\newcommand{\U}[1]{U(\mf{#1})}

\newcommand{\ev}[1]{#1 \mathrm{_{\overline{0}}}}
\newcommand{\od}[1]{#1 \mathrm{_{\overline{1}}}}
\newcommand{\supalg}[1]{\ev{#1} \oplus \od{#1}}

\newcommand{\sy}[2]{\Omega^{#1}(#2)}

\newcommand{\gen}[1]{ \langle #1 \rangle}
\newcommand{\Ind}{\op{Ind}}
\newcommand{\rel}[1]{( \mf{#1}, \ev{\mf{#1}})}
\newcommand{\frel}[1]{\mc{F}_{\rel{#1}}}
\newcommand{\ind}[1]{\U{#1} \otimes_{U(\ev{\mf{#1}})}}
\newcommand{\urel}[1]{( U(\mf{#1}), U(\ev{\mf{#1}}))}

\newcommand{\uplusrel}[1]{( U(\mf{#1}^+), U(\ev{\mf{#1}}))}

\newcommand{\F}[1]{\mc{F}_{\rel{#1}}}

\begin{document}
\title{%
Endotrivial Modules for the General Linear Lie Superalgebra }
\author{Andrew J. Talian }
\thanks{Research of the author was partially supported
by NSF grant DMS-0738586.}
\address{Department of Mathematics\\ University of Georgia \\
Athens\\ GA 30602, USA}
\curraddr{Department of Mathematics \\ Concordia College \\ Moorhead \\ MN 56562, USA}
\email{atalian@cord.edu} 
\date{June 2013}

\begin{abstract}
If $\mf{g} = \supalg{\mf{g}}$ is a Lie superalgebra over an algebraically
closed field $k$ of characteristic 0, the notion of an endotrivial module
has recently been extended to
$\mf{g}$-modules by defining $M$ to be endotrivial if
$\Hom_k(M,M) \cong k_{ev} \oplus P$
as $\mf{g}$-supermodules.  Here, $k_{ev}$ denotes the trivial module concentrated
in degree $\overline{0}$ and $P$ is a $\urel{g}$-projective supermodule.
In the stable module category, these modules form a group under 
the tensor product.  If $T(\mf{g})$ denotes the group of
endotrivial $\mf{g}$-modules, it is interesting and useful to
identify this group
for a given Lie superalgebra $\mf{g}$.  In this paper, a classification
is given in
the case where $\mf{g} = \mf{gl}(m|n)$ and it is shown that
$T(\mf{gl}(m|n)) \cong k \times \Z \times \Z_2$ and is generated by the one parameter
family of one
dimensional modules $k_\lambda$ where $\lambda \in k$, $\Omega^1(k_{ev})$,
which
denotes the first syzygy of
$k_{ev}$, and the parity  change functor.
\end{abstract}

\date{}
\maketitle

\section{Introduction}
Endotrivial modules were first defined by Dade in 1978
for $kG$-modules where $G$ is a finite group and
$k$ is a field of characteristic $p$ where $p$ divides the order of $G$.  A module
$M$ is called endotrivial if there is a $kG$-module isomorphism $\Hom_k(M,M) \cong k \oplus P$ where $k$ is the trivial module and $P$ is a projective module.  Dade's study
of this class of module arose through study of endopermutation modules in 
\cite{Dade1-1978}, and in \cite{Dade2-1978} Dade showed that, in the case when $G$ is an
abelian $p$-group, any endotrivial module
is of the form $\Omega^n(k) \oplus P$ where $\Omega^n(k)$ is the $n$th syzygy of the
trivial module $k$ and $P$ is a projective module.  Syzygies, which are sometimes
called Heller shifts or operators, are discussed in Definition \ref{D:syzygy}.

An interesting aspect of the set of endotrivial modules is that they form a group
in the stable module category where the group operation is the tensor product, $[M] + [N] = [M \otimes N]$.  Puig showed in \cite{Puig-1990} that the group
of endotrivial $kG$-modules, denoted $T(G)$,
is finitely generated for any finite group $G$.  Carlson and Th\'evenaz gave
a complete classification of $T(G)$ when $G$ is an arbitrary $p$-group in
\cite{CT-2004} and \cite{CT-2005}.
Carlson, Mazza, and Nakano have continued the study of $T(G)$ by giving
classifications when $G$ is the symmetric or alternating group
for certain cases in
\cite{CMN-2009} and
Carlson, Hemmer, and Mazza furthered those results
in \cite{CHM-2010}.

The definition of an endotrivial module has been extended beyond $kG$-modules and has been
successfully implemented and studied in a number of other areas of representation theory.
Carlson, Mazza, and Nakano have studied endotrivial modules over finite groups of Lie type
in the defining characteristic in \cite{CMN-2006} and non-defining characteristic
in \cite{CMN-2014}.  Carlson and Nakano also introduced this definition in the study
of modules for finite group schemes in \cite{CN-2009} where they prove that the
endotrivial modules for a unipotent abelian
group scheme are of the form
$\sy{n}{k} \oplus P$.  Although it is not known whether the group of endotrivial modules
over a finite group scheme is finitely generated, Carlson and Nakano proved in a
subsequent paper \cite{CN-2011} that for an arbitrary finite group scheme,
the number of isomorphism
classes of endotrivial modules of a fixed dimension is finite.

The author began the study of endotrivial modules of a Lie superalgebra
$\mf{g} = \supalg{\mf{g}}$ over an algebraically
closed field $k$ of characteristic 0 in \cite{Talian-2013} working
in the category $\F{g}$.  When $\mf{g}$ is classical $\mc{F} = \F{g}$ denotes
the category of finite dimensional $\mf{g}$-modules which are
completely reducible over $\ev{\mf{g}}$.  This is an important category
which has been of significant interest recently and has been studied
in \cite{BKN1-2006}, \cite{B-2002}, \cite{DS-2005},
\cite{S-1998}, and \cite{S-2006}, among others.
The category $\mc{F}$ has enough projectives (\cite{BKN1-2006}),
is self-injective (\cite{BKN3-2009}), meaning that a module
is projective if and only if it is injective, and
for Type I classical Lie superalgebras, e.g. $\mf{gl}(m|n)$, $\mc{F}$
is a highest weight category (\cite{BKN2-2009}).

In this context, a
$\mf{g}$-supermodule $M \in \mc{F}$ is called endotrivial
if there is a supermodule isomorphism
$\Hom_k (M,M) \cong k_{ev} \oplus P$ where $k_{ev}$ denotes the
trivial supermodule concentrated in degree $\overline{0}$ and $P$ is
a projective module in $\mc{F}$ (discussed in Section \ref{SS: Relative Category}).
As has been noted, such modules are an interesting and natural object
to study since they form a group denoted as $T(\mf{g})$ and tensoring with such modules
gives a self equivalence of the stable module category.  Thus, identifying
$T(\mf{g})$ may lead to a better understanding of
$\mc{F}$ or the Picard group of the stable module category
via techniques such as those in \cite{Balmer-2010}.

The author showed in \cite{Talian-2013} that, for a detecting subalgebra of type $\mf{e}$ or $\mf{f}$
(introduced in \cite{BKN1-2006}) whose rank is greater than 1, denoted generically as
$\mf{a}$, there is an isomorphism $T(\mf{a}) \cong \Z \times \Z_2$.
By definition, $\mf{a}$
is isomorphic to a direct sum of copies of either $\mf{q}(1)$ or $\mf{sl}(1|1)$.
The detecting subalgebras are analogous to elementary abelian subgroups
in modular representation theory in the sense that they detect cohomology
and are a natural starting point for the study of endotrivial modules in $\mc{F}$.
In the same paper, it is also shown that under certain restrictions, the
number of endotrivial modules of a fixed dimension $n$ is finite, giving
a result similar to the one mentioned in \cite{CN-2011}.  However, this statement
cannot hold in general by observing that, even for small cases like
$\mf{gl}(1|1)$, there are infinitely many one dimensional modules which are
necessarily endotrivial, forming a subgroup isomorphic to the field $k$.

This paper seeks to build on these results by giving a classification of the group
of endotrivials for the general linear Lie superalgebra, $\mf{gl}(m|n)$.  The main
result, stated in Theorem \ref{T: classification of T(g)}, is that
$$
T(\mf{gl}(m|n)) \cong k \times \Z \times \Z_2.
$$
This is achieved by
defining
an extension of $\F{g}$ for non-classical Lie superalgebras $\mf{g}$ and then
working through an intermediate parabolic subalgebra, denoted as $\mf{p}$.
Endotrivial $\F{p}$ modules are in a sense easier to understand
and in Theorem \ref{T: classification of T(p)} it is shown that there is an isomorphism
$$
T(\mf{p}) \cong k^{r+s} \times \Z \times \Z_2
$$
where $r = \min(m,n)$ and $s= |m - n|$ when $\mf{p} \subseteq \mf{gl}(m|n)$.
Furthermore, by using a geometric induction functor defined in \cite{GS-2010},
Corollary \ref{C: Res is inj} shows there is an
injection
$$
T(\mf{gl}(m|n)) \hookrightarrow T(\mf{p})
$$
and the image is computed
directly, yielding the main theorem.  The classification of $T(\mf{p})$
results from studying the restriction map $T(\mf{p}) \rightarrow T(\mf{f})$ 
given by $M \mapsto M|_{\mf{f}}$ and identifying the kernel.  This is
more approachable than restriction from $\mf{gl}(m|n)$ to $\mf{f}$
because $\mf{p}$ has a smaller and more easily handled set of weights.

\section{Notation and Preliminaries}
\subsection{The Distinguished Parabolic} \label{SS: dist para}

In \cite{BKN1-2006}, the category $\mc{F}_{(\mf{g},\mf{t})}$ is defined for
a classical Lie superalgebra $\mf{g}$ where $\mf{t} \subseteq \mf{g}$ is a subalgebra
of $\mf{g}$.  In this paper, we wish to consider a compatible extension of this notion
beyond the classical case (Section \ref{SS: Relative Category}).
To motivate this, consider
the following.

The classical Lie superalgebra $\mf{gl}(m|n)$ can be defined as
$(m+n) \times (m+n)$ matrices with standard basis vectors
$e_{i,j}$ where $1 \leq i,j \leq m+n$.  The usual grading
is that the even part is defined to be matrices where the only nonzero
entries are in the
$m \times m$ and $n \times n$ block diagonal and the odd part
is defined to be matrices where the only nonzero entries are in the off block
diagonal, i.e.
\[ \ev{(\mf{gl}(m|n))} = \left( \begin{array}{c|c}
		A & 0   \\ \hline
		0 & D  \\
		\end{array} \right)	\quad	
	\od{(\mf{gl}(m|n))} =			
	\left( \begin{array}{c|c}
			0 & B   \\ \hline
			C & 0  \\
			\end{array} \right)
\]
where $A \in M_{m,m}(k)$, $D \in M_{n,n}(k)$, $B \in M_{m,n}(k)$,
and $C \in M_{n,m}(k)$.
It can be verified directly that
$\mf{gl}(m|n) \cong \ev{(\mf{gl}(m|n))} \oplus \od{(\mf{gl}(m|n))}$
as a $\Z_2$ graded algebra via matrix multiplication.
If $Z \in (\mf{gl}(m|n))_i$ is
homogeneous then define $|Z| = i$, and define
a bilinear multiplication in $\mf{gl}(m|n)$  by
the super commutator bracket 
$[X,Y] := XY - (-1)^{|X||Y|} YX$ for homogeneous elements $X,Y
\in \mf{gl}(m|n)$.  The definition is extended to all elements
by linearity and this standard construction gives the matrices
the structure of a Lie superalgebra under the bracket operation.

Let $\mf{p}$ denote the \emph{distinguished parabolic subalgebra}
of $\mf{gl}(m|n)$ defined as follows.
Let $\ev{\mf{p}} \subseteq \ev{(\mf{gl}(m|n))} \cong \mf{gl}(m) \oplus
\mf{gl}(n)$ be generated by the upper triangular matrices of $\mf{gl}(m)$ and
$\mf{gl}(n)$.
Define
$\od{\mf{p}} \subseteq \od{(\mf{gl}(m|n))}$ as the
$m \times n$ and $n \times m$
matrices whose entries are all on or above the odd diagonal.
That is, if $B$ and $C$ are as above, in the standard basis
vectors $\od{\mf{p}}$ is generated by $B' \in M_{m,n}(k)$ where
the only nonzero entries are $e_{i,j}$ where $j \geq i + m$
and $C' \in M_{n,m}(k)$ where $j +m \geq i$.  Then $\mf{gl}(m|n)$
and $\mf{p}$ share a maximal torus $\ev{\mf{t}}$
of the (even) diagonal
matrices.

Note that $\mf{p}$ is not classical however,
and in fact $\ev{\mf{p}}$ is a solvable Lie algebra.  This
requires an extension of the definition of
$\mc{F}_{(\mf{g}, \mf{t})}$ and while the following is written in
a general context, $\mf{p}$ is the primary example to keep
in mind.

\subsection{Relative Projectivity}

Before defining the category $\F{g}$,
the notion of relatively projective modules is considered, as
detailed in \cite[Appendix D]{Kumar-2002}.  If $G$ is a  superalgebra
and $H \subseteq G$ a subsuperalgebra, a sequence of
$G$-supermodules
$$
\cdots \rightarrow M_{i-1} \xrightarrow{f_{i-1}} M_i \xrightarrow{f_i}
	M_{i+1} \rightarrow \cdots
$$
where each $f_i$ is even, i.e. preserves the grading of
the modules,
is called $(G,H)$-exact if it is exact as a sequence as
$G$-supermodules and when the sequence is
considered as $H$-supermodules, $\ker f_i$ is a direct summand
of $M_i|_{H}$ for all $i$.  A $G$-supermodule is called
$(G,H)$-projective if for any $(G,H)$-exact sequence
$$
0 \rightarrow M_1 \xrightarrow{f} M_2 \xrightarrow{g} M_3 \rightarrow 0
$$
and $G$-supermodule map $h: P \rightarrow M_3$ there
is a $G$-supermodule map $\tilde{h} : P \rightarrow M_2$
such that $g \circ \tilde{h} = h$.  Note that any projective $G$-module
is necessarily $(G,H)$-projective.  Relatively injective
modules are defined in a dual way.  

The particular case of interest will be
$(U(\mf{g}),U(\ev{\mf{g}}))$-projective modules.  By
\cite[Lemma D.2]{Kumar-2002}, any $U(\mf{g})$-supermodule $M$ has a
$(U(\mf{g}),U(\ev{\mf{g}}))$-projective module which surjects
onto $M$ given by
$U(\mf{g}) \otimes_{U(\ev{\mf{g}})} M$.  Dually, any such $M$
also has an injective module into which $M$ injects given by
$\Hom_{U(\ev{\mf{g}})}(U(\mf{g}), M)$.

\subsection{The Relative Category} \label{SS: Relative Category}

When $\mf{g}$ is not classical, we
define $\mc{F}_{\rel{g}}$ to be finite dimensional $\U{g}$-modules which are completely
reducible over a fixed maximal
semisimple torus $\ev{\mf{t}} \subseteq \ev{\mf{g}}$, and
the morphisms are all even $\U{g}$-module homomorphisms.
Note, as in
\cite{BKN1-2006}, \cite{BKN3-2009}, and \cite{Talian-2013} the projective
(respectively, injective) objects in
this category will be $\urel{g}$-projective (respectively, $\urel{g}$-injective) modules.
Furthermore, we can also define $\Ext^i_{\rel{g}}(M,N)$ and
$\H^i (\mf{g}, \ev{\mf{g}};  M)$ whose constructions are given in \cite{BKN1-2006}.

Since the category $\mc{F}_{\rel{f}}$ is used extensively in
\cite{Talian-2013}, for the sake of compatibility,
we make the following assumptions on $\mc{F}_{\rel{g}}$ when $\mf{g}$ is
a stable Lie superalgebra (see \cite{BKN1-2006} for the definition).
For stable $\mf{g}$,
there exists a detecting subalgebra $\mf{f} \subseteq \mf{g}$ with
maximal torus $\mf{t_f} \subseteq \mf{f}$.
Let $\ev{\mf{t}} \subseteq \mf{g}$
be a torus for the Lie algebra $\ev{\mf{g}}$ such that
$\mf{t_f} \subseteq \ev{\mf{t}}$.  Then $\mc{F}_{\rel{g}}$ modules are assumed to be
completely reducible over the torus $\ev{\mf{t}}$ such that
$\mf{t_f} \subseteq \ev{\mf{t}}$.

A few preliminary results are given to establish the theory of
endotrivial modules in 
$\mc{F}_{\rel{g}}$, 
which is by convention denoted simply as $\mc{F}$ when there is no ambiguity.
The following proposition
gives a concrete description of the projective and injective modules
in the category, as well as some important properties of $\mc{F}$.

\begin{proposition} \label{P: main}
Let $M$, $P$, and $I$ be modules in $\F{g}$.
\begin{enumerate}[(a)]
\item A module P is $\urel{g}$-projective if and only if it is
a direct summand of $U(\mf{g}) \otimes_{U(\ev{\mf{g}})} N$ for some
$U(\ev{\mf{g}})$-module $N$.
\label{P: rel proj are summands of induced}

\item For the module $M$, there exists a projective
module $P$ and an injective module $I$ such that there are homomorphisms of $\mc{F}$ modules
$\pi: P \twoheadrightarrow M$ and $\iota: M \hookrightarrow I$.
\label{C: proj covers exist}

\item A module $P$ is projective in $\mc{F}$
if and only if it is an injective module in $\mc{F}$.
\label{C: self inj}
\end{enumerate}
\end{proposition}
\begin{proof}
For (\ref{P: rel proj are summands of induced}),
first assume that $P$ is projective in $\mc{F}$.  The
following sequence is, by construction,
$\urel{g}$-exact
$$
\begin{tikzpicture}[start chain] {
	\node[on chain] {$0$};
	\node[on chain] {$\ker \mu$} ;
	\node[on chain] {$\ind{g} P|_{\ev{\mf{g}}}$};
	\node[on chain, join={node[above, font=\scriptsize] {$\mu$}}] {$P$};
	\node[on chain] {$0$}; }
\end{tikzpicture}
$$
and is split by using the $\urel{g}$-projectivity of $P$ to extend the
identity map on $P$ in the standard way.

Now, let $P$ be a direct summand of $U(\mf{g}) \otimes_{U(\ev{\mf{g}})} N$
for some $U(\ev{\mf{g}})$-module $N$.  Then
$$
\Ext^1_{(\mf{g}, \ev{\mf{g}})}(P, R) \hookrightarrow 
\Ext^1_{(\mf{g}, \ev{\mf{g}})}(U(\mf{g}) \otimes_{U(\ev{\mf{g}})} N, R) =
\Ext^1_{(\ev{\mf{g}}, \ev{\mf{g}})}(N,R) = 0
$$
for any module $R$ in $\mc{F}$.  Thus, $P$ is
$\urel{g}$-projective and so it is projective in $\mc{F}$.

Part (\ref{C: proj covers exist}) follows from
\cite[Lemma D.2]{Kumar-2002} by noting that the extra condition
of complete reducibility holds and the proof given in \cite[Propositions 2.2.2]{BKN3-2009} holds for
$\mc{F}$ which proves (\ref{C: self inj}).
\end{proof}

\begin{definition} \label{D:syzygy}
Let $\mf{g}$ be a
Lie superalgebra and let $M$ be a module in
$\mc{F}_{\rel{g}}$.  Let
$P$ be a minimal projective module in $\mc{F}$ which surjects on to $M$ (called the projective
cover), with the map
$
\psi : P \twoheadrightarrow M.
$
The \define{first syzygy of $M$} is defined to be $\ker \psi$ and is denoted by
$\Omega^1_{\mf{g}}(M)$.  This is also referred to as a
Heller shift (or Heller operator) in some literature.  Inductively, define
$\Omega^{n+1}_{\mf{g}}(M) := \Omega^1_{\mf{g}}(\Omega^n_{\mf{g}}(M))$.

Similarly, given $M$, let $I$ be the injective hull of $M$ with the
inclusion
$
\iota : M \hookrightarrow I,
$
and define $\Omega\inv_{\mf{g}} (M) := \coker \iota$.  This is extended
by defining $\Omega^{-n-1}_{\mf{g}}(M) := \Omega\inv_{\mf{g}} ( \Omega^{-n}_{\mf{g}}(M))$.

Finally, define $\Omega^0_{\mf{g}}(M)$ to be the compliment of the largest $\urel{g}$-projective direct
summand of $M$.  In other words, we can write $M = \Omega^0_{\mf{g}}(M) \oplus Q$
where $Q$ is projective in $\mc{F}$ and maximal with
respect to this property.
Thus, the \define{$n$th syzygy of $M$} is defined for any
integer $n$.
\end{definition}

\begin{convention}
When there is no ambiguity, $\Omega^n_{\mf{g}}(M)$ may
be denoted as $\Omega^n(M)$.
\end{convention}

\subsection{Endotrivial Modules}

With a better understanding of projective modules in $\mc{F}$,
we now define the object of interest in this paper.

\begin{definition}
A module in $\F{g}$ is called endotrivial if
$\End_k(M) \cong k_{ev} \oplus P$ as $U(\mf{g})$-modules for
some projective module $P$ in $\mc{F}$.
\end{definition}

This definition is equivalent to defining $M$ to be endotrivial
if $M \otimes M^* \cong k_{ev} \oplus P$ by using the isomorphism
$\Hom(V,W) \cong W \otimes V^*$.
One of the interests in studying endotrivial modules is that they
form a group in the stable module category.
\begin{definition}
Given a category of modules, $\mc{F}$,
consider the category with the same objects
as the original category and an equivalence relation on the morphisms
given by $ f \sim g$ if $f - g$ factors through a projective module in
$\mc{F}$.
This is called the \define{stable module category} of $\mc{F}$ and is denoted
by $\operatorname{Stmod}(\mc{F})$.
\end{definition}

\begin{definition}
Let $\mf{g}$ be a Lie superalgebra.  The set of endotrivial 
modules in $\op{Stmod}(\mc{F}_{\rel{g}})$
$$
T(\mf{g}) := \left\{ [M] \in \operatorname{Stmod}(\mc{F}) \st M \otimes M^* \cong
k_{ev} \oplus P_M \text{ for some $P_M$ which is projective in $\mc{F}$} \right\}
$$
forms a group in the stable module category of $\mc{F}$
under the operation $[M] + [N] := [M \otimes N]$.  This group is called
the \emph{group of endotrivial $\mf{g}$-supermodules}.
\end{definition}
More details on syzygies and this group are given in \cite{Talian-2013}.  One such
observation is that if $M$ is any endotrivial module, then $\Omega^n(M)$ is
endotrivial as well for any $n \in \N$.  An additional result is stated here
relating syzygies relative to different Lie superalgebras and
will be useful throughout this work.

\begin{lemma} \label{L: restriction commutes}
Let $\mf{g}$ be a Lie superalgebra with torus $\mf{t_g}$.
Let $\mf{h} \subseteq \mf{g}$ be a Lie subalgebra with torus $\mf{t_h}$
such that $\mf{t_h} \subseteq \mf{t_g}$ and that for each projective
module $Q$ in $\mc{F}_{\rel{g}}$,
$Q|_{\mf{h}}$ is projective
in $\mc{F}_{\rel{h}}$.
Let $M$ be a module in $\mc{F}_{\rel{g}}$,
then $\Omega^n_{\mf{g}}(M)|_{\mf{h}} \cong \Omega^n_{\mf{h}}(M|_{\mf{h}})
\oplus P$
for all $n \in \Z$ where $P$ is a projective module in $\mc{F}_{\rel{h}}$.
\end{lemma}
\begin{proof}
Let $M$ be as above and let
$$
\begin{tikzpicture}[start chain] {
	\node[on chain] {$0$};
	\node[on chain] {$\Omega^1_{\mf{g}}(M)$} ;
	\node[on chain] {$Q$};
	\node[on chain] {$M$};
	\node[on chain] {$0$}; }
\end{tikzpicture}
$$
be the short exact sequence of modules in $\F{g}$ defining
$\Omega^1_{\mf{g}}(M)$.  Then
$$
\begin{tikzpicture}[start chain] {
	\node[on chain] {$0$};
	\node[on chain] {$\Omega^1_{\mf{g}}(M)|_{\mf{h}}$} ;
	\node[on chain] {$Q|_{\mf{h}}$};
	\node[on chain] {$M|_{\mf{h}}$};
	\node[on chain] {$0$}; }
\end{tikzpicture}
$$
is an exact sequence
and the module $Q|_{\mf{h}}$ is projective in $\mc{F}_{\rel{h}}$
(although perhaps not minimal).  Then by definition,
$\Omega^1_{\mf{g}}(M)|_{\mf{h}} \cong \Omega^1_{\mf{h}}(M|_{\mf{h}})
\oplus P$.
This argument applies to $\Omega\inv_{\mf{g}}(M)$ as well and so by induction,
$\Omega^n_{\mf{g}}(M)|_{\mf{h}} \cong \Omega^n_{\mf{h}}(M|_{\mf{h}})
\oplus P$
for all $n \in \Z$
\end{proof}

\section{The Distinguished Parabolic} \label{S: glmn}

As noted,
the primary purpose of extending the definition of $\mc{F}_{\rel{g}}$ beyond classical
Lie superalgebras, is to consider the case for $\mf{p}$, 
the distinguished parabolic
subalgebra of $\mf{gl}(m|n)$ defined in Section
\ref{SS: dist para}.
There is a relationship between the groups
$T(\mf{gl}(m|n))$ and $T(\mf{p})$ and understanding $T(\mf{p})$ will eventually
lead to a classification of $T(\mf{gl}(m|n))$, the main goal of this paper.

Note that if $\mf{f}$ is defined to be the subalgebra of $\mf{p}$ generated
by elements which are strictly on the odd diagonal (plus a
torus of dimension $|m-n|$ which is described in the proof of
Theorem \ref{T: classification of T(p)}), then $\mf{f} \subseteq \mf{p} \subseteq \mf{gl}(m|n)$
and $\mf{t_f} \subseteq \mf{t_p} = \ev{\mf{t}}$.
Given this set up, we will relate $T(\mf{f})$, $T(\mf{p})$, and
$T(\mf{gl}(m|n))$.

The main reasons to study $\mf{p}$ are that the set of weights relative to the torus
for $\mf{f}$ are well behaved and that it has a $\Z$ grading which is consistent with
the $\Z_2$ grading.
This $\Z$ grading allows for results analogous to those in \cite{BKN3-2009} and
\cite{LNZ-2011} to be
extended to $\mf{p}$.

In the following two sections, we exploit this $\Z$ grading to derive results about projectivity when restricting
from $\mf{gl}(m|n)$ to $\mf{p}$ to $\mf{f}$.
First, we establish that restriction to each of these
subalgebras takes projectives to projectives in order to have well defined
maps between the groups which are defined in the stable module category.  Second, it is
shown that if a module in $\F{p}$ is projective when restricted to $\F{f}$, then it
is projective in $\F{p}$ as well.

\subsection{Restriction and Projectivity}

Because $\mf{g} = \mf{gl}(m|n)$ is a Type I Lie superalgebra,
$\mf{g}$ has a $\Z$ grading of the form
$\mf{g} = \mf{g}_{-1} \oplus \mf{g}_0 \oplus \mf{g}_1$ which is consistent
with the standard $\Z_2$ grading.  This gives a consistent $\Z$ grading on
$\mf{p} \subseteq \mf{g}$
by defining $\mf{p}_i = \mf{p} \cap \mf{g}_i$ for $i \in \Z$, and so
$\mf{p} = \mf{p}_{-1} \oplus \mf{p}_0 \oplus \mf{p}_1$.  Given this grading,
define $\mf{p}^+ := \mf{p}_0 \oplus \mf{p}_1$ and
$\mf{p}^- := \mf{p}_{-1} \oplus \mf{p}_0$.  Similarly, we may
decompose $\mf{f} = \mf{f}_{-1} \oplus \mf{f}_0 \oplus \mf{f}_1$ and
define $\mf{f}^+ :=
\mf{f}_0 \oplus \mf{f}_1$ and $\mf{f}^- := \mf{f}_{-1} \oplus \mf{f}_0$.

Following the work in \cite{BKN3-2009}, define $\mc{F}(\mf{p}_{\pm 1})$
to be the category of finite dimensional $\mf{p}_{\pm 1}$-modules.
For the objects in $\mc{F}(\mf{p}_{\pm 1})$, define
the support variety
$\mc{V}_{\mf{p}_{\pm 1}}(M)$ as in \cite{BKN3-2009} and the rank variety
$$
\mc{V}^{\text{rank}}_{\mf{p}_{\pm 1}}(M) = 
\{x \in \mf{p}_{\pm 1} \st M \text{ is not projective as
a $U(\gen{x})$-module}  \} \cup \{0 \} .
$$
Since $\mf{p}_1$ and $\mf{p}_{-1}$ are both abelian Lie superalgebras,
both are well defined and identified by the canonical isomorphism
detailed in \cite{BKN1-2006}.

Consider $X(\ev{\mf{t}}) \subseteq \ev{\mf{t}}^*$,
the set of weights relative to
a fixed maximal torus $\ev{\mf{t}} \subseteq \ev{\mf{p}}$.
It will be very useful to have a partial ordering on these weights.
Let $d = \dim \ev{\mf{t}}$.
If we fix the dual basis of $\ev{\mf{t}}$ to be the basis for 
$X(\ev{\mf{t}})$, the weights can be parameterized by the set
$k^{d}$ so
any $\lambda \in X(\ev{\mf{t}})$ can be though of as an ordered
$d$-tuple
$(\lambda_1, \dots, \lambda_{d})$.  For two weights
$\lambda = (\lambda_1, \dots, \lambda_{d})$ and
$\mu = (\mu_1, \dots, \mu_{d})$, we say that
$\lambda \geq \mu$ if and only if for each $k = 1, \dots, d$,
\begin{equation} \label{E:Dominance ordering}
\sum\limits_{i=1}^{k} \lambda_i \geq \sum\limits_{i=1}^{k} \mu_i
\end{equation}
and equality holds if and only if $\lambda = \mu$.
This ordering is referred to as the \emph{dominance ordering}
and will
allow the use of highest weight theory.

A module $M \in\mc{F}_{(\mf{p}, \ev{\mf{p}})}$ is called a \emph{highest weight module} if
in the weight decomposition
$M \cong \bigoplus_{\lambda \in X(\ev{\mf{t}})} M_{\lambda}$, there
exists a weight $\lambda_0$ such that $\lambda_0 \geq \mu$ for each
nonzero weight space $M_{\mu}$ of $M$.  

\begin{proposition}
If $S \in \mc{F}_{(\mf{p}, \ev{\mf{p}})}$  is simple, then $S$ is a highest
weight module in $\mc{F}_{(\mf{p}, \ev{\mf{p}})}$.
\end{proposition}
\begin{proof}
Because $S$ is finite dimensional, there exists a weight $\lambda_0 \in X(\ev{\mf{t}})$ such
that $\mu \ngtr \lambda_0$ for all nonzero weight spaces $S_{\mu}$ of $S$.
Note that this means all weights are either less than or equal to or not
comparable to $\lambda_0$.

For any element $p$ of $\mf{p}_1 $,
$p.S_{\lambda} \subseteq S_{\mu}$ implies that $\mu > \lambda$ in
$X(\ev{\mf{t}})$.  This
yields that $p.S_{\lambda_0} = 0$ for any $p \in \mf{p}_1$.  Since $S$ is
simple,
for $v \in S_{\lambda_0}$, $v$ generates $S$ and $p.v = 0$.

Thus, $S = U(\mf{p}_{-1})U(\mf{p}^+).v$ but since any element
of $\mf{p}^+$ either
stabilizes or kills $v$, it follows that $S = U(\mf{p}_{-1}).v$.  It is
now clear, because $v \in S_{\lambda_0}$ that any element of $S_{\mu} \neq 0$
is equal to $cy.v$ for some $y \in U(\mf{p}_{-1})$ and $c \in k$, where
$\lambda_0 \geq \mu$.  Thus $S$ is a highest weight module.
\end{proof}

Since $\ev{\mf{p}}$ is a solvable Lie algebra, the only simple
modules are one dimensional modules $k_\lambda$ where the torus
acts by weight $\lambda \in X(\ev{\mf{t}})$.
Because $\mf{p}_1 \subseteq \mf{p}^+$ and $\mf{p}_{-1} \subseteq \mf{p}^-$
are ideals, the module $k_\lambda$ can be considered as a simple
$\mf{p}^\pm$-module by
inflation via the canonical quotient map
$\mf{p}^{\pm} \twoheadrightarrow \ev{\mf{p}}$.  By construction,
$\mf{p}_{1}$ and $\mf{p}_{-1}$ act by 0 on $k_\lambda$.  Define
$$
K(\lambda) = U(\mf{p}) \otimes_{U(\mf{p}^+)} k_\lambda
\quad \text{and} \quad
K^-(\lambda) = \Hom_{U(\mf{p}^-)} (U(\mf{p}), k_\lambda)
$$
to be the Kac module and the dual Kac module, respectively.

The Kac module $K(\lambda)$ has several useful properties.
First, by construction it is a highest weight module in
$\mc{F}_{(\mf{p}, \ev{\mf{p}})}$.
Since $K(\lambda)$ is generated by a highest weight vector,
it has a simple head.
Also, if $S$ is any simple module in
$\mc{F}_{(\mf{p}, \ev{\mf{p}})}$ where $S$ has
highest weight $\lambda$ for some weight $\lambda$ of $S$,
and $v \in S_\lambda$, 
there is a surjective homomorphism $K(\lambda) \twoheadrightarrow S$
given by $u \otimes 1 \mapsto u.v$.

Furthermore,
$K(\lambda)/ \op{Rad}(K(\lambda)) \cong S$ and is denoted $L(\lambda)$.
Note that this surjective homomorphism is in fact
valid for any highest weight module and
in this sense, the Kac module is universal.

Dually, simple modules in $\F{p}$ are lowest weight modules and
if $L(\lambda)$ has lowest weight $\mu$,
then $K^-(\mu)$ has a simple socle which is isomorphic to
$L(\lambda)$ as well and $\mu$ is the lowest weight of $K^-(\mu)$.

Now we define two useful filtrations of a module $M$ in $\mc{F}_{(\mf{p}, \ev{\mf{p}})}$.
$M$ is said to admit a Kac filtration if there is a filtration
$$
\{0 \} = M_0 \subsetneq M_1 \subsetneq \dots \subsetneq M_t = M
$$
of the module $M$ such that for $i = 1, \dots t$,
$M_i / M_{i-1} \cong K(\lambda_i)$
for some $\lambda_i \in X(\ev{\mf{t}})$.  Similarly,
if $M$ has a filtration as above such that for $i = 1, \dots t$,
$M_i / M_{i-1} \cong K^-(\lambda_i)$, then $M$ is said to admit a dual
Kac filtration.

By the same reasoning in \cite{BKN3-2009}, modules in 
$\mc{F}_{(\mf{p}, \ev{\mf{p}})}$ satisfy
the following.

\begin{theorem} \label{T: Kac filtration conditions}
Let $M$ be a module in
$\mc{F}_{(\mf{p}, \ev{\mf{p}})}$.
Then the following are equivalent.
\begin{enumerate}
\item $M$ has a Kac filtration;

\item $\Ext^1_{\F{p}} (M, K^-(\mu)) = 0$ for all
$\mu \in X(\ev{\mf{t}})$;

\item $\Ext^1_{\mc{F}(\mf{p}_{-1})} (M, k) =0$;

\item $\mc{V}_{\mf{p}_{-1}} (M) = 0$.
\end{enumerate}
\end{theorem}
\begin{theorem} \label{T: dual Kac filtration conditions}
Let $M$ be a module in $\mc{F}_{(\mf{p}, \ev{\mf{p}})}$.  Then the following are equivalent.
\begin{enumerate}
\item $M$ has a dual Kac filtration;

\item $\Ext^1_{\F{p}} (K(\mu), M) = 0$ for all $\mu \in X(\ev{\mf{t}})$;

\item $\Ext^1_{\mc{F}(\mf{p}_{1})} (k, M) =0$;

\item $\mc{V}_{\mf{p}_{1}} (M) = 0$.
\end{enumerate}
\end{theorem}

These two theorems can be used to show the following powerful condition
relating projectivity in $\mc{F}_{\rel{p}}$ and the support varieties
of $\mf{p}_{\pm 1}$.

\begin{theorem} \label{T: p proj equiv to tilting}
Let $M$ be in $\mc{F}_{(\mf{p}, \ev{\mf{p}})}$.  Then $M$
is projective in $\mc{F}_{(\mf{p}, \ev{\mf{p}})}$ if and only
if $\mc{V}_{\mf{p}_1}(M) = \mc{V}_{\mf{p}_{-1}}(M) = \{ 0 \}$.
\end{theorem}

\begin{corollary} \label{C: T(g) to T(p)}
A projective module in
$\mc{F}_{(\mf{g}, \ev{\mf{g}})}$ is also projective in
$\mc{F}_{(\mf{p}, \ev{\mf{p}})}$ and thus, there is a
well defined map
$$\op{res}^{T(\mf{g})}_{T(\mf{p})}: T(\mf{g}) \rightarrow T(\mf{p})$$ given by $M \mapsto M|_{\mf{p}}$.  Moreover, this map is a homomorphism of groups.
\end{corollary}
\begin{proof}
Let $P$ be a projective module in
$\mc{F}_{(\mf{g}, \ev{\mf{g}})}$.  Then by
\cite[Theorem 3.5.1]{BKN3-2009},
$\mc{V}_{\mf{g}_1}(M) = \mc{V}_{\mf{g}_{-1}}(M) = \{ 0 \}$.  Using the rank variety description, we see that
$\mc{V}_{\mf{p}_1}(M) \subseteq \mc{V}_{\mf{g}_{1}}(M) = \{ 0 \}$ and
$\mc{V}_{\mf{p}_{-1}}(M) \subseteq \mc{V}_{\mf{g}_{-1}}(M) = \{ 0 \}$,
and so by Theorem \ref{T: p proj equiv to tilting}, $M|_{\mf{p}}$
is projective in $\mc{F}_{(\mf{p}, \ev{\mf{p}})}$.

With this conclusion, the restriction map now descends to a well defined
map on each of the respective stable module categories, and in particular,
if $M \in \mc{F}_{(\mf{g}, \ev{\mf{g}})}$, such that
$M \otimes M^* \cong k_{ev} \oplus P$, then
$(M \otimes M^*)|_{\mf{p}} \cong k_{ev} \oplus P|_{\mf{p}}$.

Furthermore, since restriction commutes with the tensor product, this is also a group
homomorphism.
\end{proof}

The following maps, which will be useful in Section
\ref{S: classification of T(g)}, now follow easily.

\begin{proposition} \label{P: b proj implies f proj}
Let $M$ be a projective module in $\F{p}$.
Then $M|_{\mf{p}^+}$ and $M|_{\mf{p}^-}$
and $M|_{\mf{f}}$ are all
projective in their respective categories.
\end{proposition}
\begin{proof}  Let $M$ be as above.
Then, by Proposition \ref{P: main}
(\ref{P: rel proj are summands of induced}), $M$ is a summand of
$\ind{p} N$ for some $U(\ev{\mf{p}})$-module $N$.  Because
\begin{gather*}
\ind{p} N \cong U(\mf{p}^+)U(\mf{p}_{-1}) \otimes_{U(\ev{\mf{p}})} N
\cong U(\mf{p}^+) \otimes_{U(\ev{\mf{p}})} [ U(\mf{p}_{-1}) \otimes N]
\end{gather*}
where the second isomorphism is given on basis elements by
$$
u^+ u_{-1} \otimes_{U(\ev{\mf{p}})} n \mapsto
	u^+ \otimes_{U(\ev{\mf{p}})} [u_{-1} \otimes n],
$$
any summand of $\ind{p} N$ is also a summand of
$U(\mf{p}^+) \otimes_{U(\ev{\mf{p}})} N'$ for some $\ev{\mf{p}}$-module
$N'$.  Thus, if $M$ is projective in
$\mc{F}_{(\mf{p}, \ev{\mf{p}})}$, $M$ is also
$\uplusrel{p}$-projective as well, or projective in
$\mc{F}_{(\mf{p}^+, \ev{\mf{p}})}$.
By a similar argument, $M$ is also
projective in $\mc{F}_{(\mf{p}^-, \ev{\mf{p}})}$.

Since the rank varieties
$\mc{V}_{\mf{p}_1}(M)$ and $\mc{V}_{\mf{p}_{-1}}(M)$ measure projectivity,
both varieties are $\{0\}$.
Additionally,
$\mc{V}_{\mf{f}_{\pm 1}}(M) \subseteq \mc{V}_{\mf{p}_{\pm 1}}(M) = 0$,
and so by \cite[Theorem 3.5.1]{BKN3-2009}, $M|_\mf{f}$ is projective
in $\mc{F}_{\rel{f}}$.
\end{proof}

\begin{corollary} \label{C: T(p) to T(f)}
A projective module in
$\mc{F}_{(\mf{p}, \ev{\mf{p}})}$ is also projective in
$\mc{F}_{(\mf{f}, \ev{\mf{f}})}$ and thus, there is a
well defined map
$$\op{res}^{T(\mf{p})}_{T(\mf{f})}: T(\mf{p}) \rightarrow T(\mf{f})$$ given by $M \mapsto M|_{\mf{f}}$.  Moreover, this map is a homomorphism of groups.
\end{corollary}

\subsection{Detecting Projectivity}
The second primary result about the relationship between $\F{p}$ and $\F{f}$ is that
if a module in $\F{p}$ is projective when restricted to $\F{f}$, then it is also
projective in $\F{p}$.  This result is derived from \cite{LNZ-2011} by observing that
the triangular decomposition of $\mf{p}$ will afford the same results as those given
for $\mf{g}$ a classical Type I Lie superalgebras.

One of the steps in the original
proof relies on an invariant
theory result about the action of the reductive group $G_0$ where
$\op{Lie}(G_0) = \mf{g}_0$ on the subalgebra
$\mf{f}_{\pm 1}$.
Namely that if $\phi \in \Hom_{G_0}(\mf{g}_{\pm 1}, k)$
such that $\phi|_{\mf{f}_{\pm 1}} \equiv 0$ then $\phi \equiv 0$.
This result is shown for $\mf{p}$ directly in the following lemma.

\begin{lemma} \label{L: f zero implies p zero}
Let $\mf{p}$ be the distinguished parabolic subalgebra with
triangular decomposition $\mf{p}_{-1} \oplus \mf{p}_0 \oplus
\mf{p}_{1}$ and $P_0$ be the algebraic group such that
$\op{Lie}(P_0) = \mf{p}_0$.  Let
$\mf{f}_i = \mf{f} \cap \mf{p}_i$ and
$\psi \in \Hom_{P_0}(\mf{p}_{\pm 1}, k)$
such that $\psi|_{\mf{f}_{\pm 1}} \equiv 0$.  Then
$\psi \equiv 0$.
\end{lemma}
\begin{proof}
Without loss of generality, the proof is given for
$\mf{f}_1$ where $\mf{p} \subseteq
\mf{gl}(m|n)$ with $n > m$ as all other cases follow
in a similar way.

In this instance, consider the given bases for these
subalgebras.  As defined in \ref{SS: dist para},
$\mf{f}_1 \subseteq \mf{p}_1$ can be thought of as
elements of $M_{m,n}(k)$ embedded in the block upper
triangular corner of $(m+n) \times (m+n)$ matrices.  So if
$0_{i,j}$ is an $i \times j$ matrix of zeros, and
\[ M =
\left(
\begin{array}{c|c}
  0_{m,m} & N \\   \hline
  0_{n,m} & 0_{n,n} \\
\end{array}
\right)
\]
where $N$ is of the form
\[ 
\left(
\begin{array}{ccccccc}
  x_1 & a_2 & a_3 & . & . & . & a_n \\ 
  0   & x_2 & b_3 & . & . & . & b_n \\
  0   & 0   & x_3 & . & . & . & c_n \\
  \vdots & \vdots  & & \ddots & & & \vdots \\
  0 & 0 & \dots & 0 & x_m & \dots & d_n
\end{array}
\right)
\]
then $M \in \mf{f}_1$ if the only (possibly) nonzero entries are
the $x_i$'s, i.e. on the odd diagonal,
and $M$ is in $\mf{p}_1$
if arbitrary variables shown above are allowed to be nonzero.

Now let $\psi$ be as above and for a fixed element
$P \in \mf{p}_1$, let $X$ be the matrix
whose $x_i$ entries are those in $P$ and all others zero, $A$ be
the matrix whose $a_i$ entries are those in $P$ (all others
zero), etc. and so
we can decompose $P$ by writing $P = X + A + B + \dots + D$.
By definition, $\psi(X) = 0$

We proceed by contradiction and an iterated argument using the
number of rows $m$.  So without loss of generality,
assume that $\psi(P) \neq 0$ and that the matrix $A \neq 0$
(otherwise proceed to the next iteration).

If $T_i$ denotes (not strictly)
upper triangular $i \times i$ matrices, then
$P_0 \cong T_m \oplus T_n$ and the action on $\mf{p}_1$
is given by $(G,H)\cdot M = GMH\inv$ for $(G,H) 
\in P_0$ and $M \in \mf{p}_1$.
By assumption, $\psi$ is a $P_0$ invariant
function so the action of $P_0$
on $\mf{p}_1$ does not change the value of the function.
Let $I_i(j,c)$ denote the $i \times i$ identity matrix
where the $j$th diagonal entry is replaced by the constant
$c \in k$
and consider the action of $I_{1,c} := (I_m(1,c),I_n(1,c))$ on the
element $P$, where $0 < c < 1$.
By construction $I_{1,c} \cdot P
=X + cA + B + \dots + D$ and by iterating the action $\ell$
times,
$(I_{1,c})^\ell \cdot P = X + c^\ell A + B + \dots + D$.  Since
$\psi$ is $P_0$ invariant,
$$
\psi((I_{1,c})^\ell \cdot P) = \psi(P) = c' \neq 0
$$
and so $(I_{1,c})^\ell \cdot P \in \psi\inv(c')$ for all $\ell > 0$.
Furthermore, $c' \in k$ is a closed set and $\psi$ is
continuous  so
$\psi\inv(c')$ is closed in the Zariski topology of $\mf{p}_1$
and contains its limit points under the action of $P_0$.
We conclude that
$$
X + B + \dots + D = \lim\limits_{\ell \to \infty} (I_{1,c})^\ell
\cdot P \in \psi\inv(c')
$$
and so $\psi(X + B + \dots + D) = c'$.

Now this argument may be repeated by considering the
action of $I_{2,c}$ on $X + B + \dots + D$, and so on
until the action of $I_{m,c}$ on $X + D$ yields that
$X  = \lim\limits_{\ell \to \infty} (I_{n,c})^\ell
\cdot (X + D) \in \psi\inv(c')$, and thus $\psi(X) = c' \neq 0$.
This is a contradiction and we
conclude that the assumption was false.  So
$\psi(P) = 0$ for any $P \in \mf{p}_1$ and $\psi \equiv 0$.
\end{proof}

\begin{theorem}
For all $M \in \F{p}$ and $n \neq 0$ the restriction map
$$
\H^n(\mf{p}, \ev{\mf{p}}, M) \rightarrow \H^n(\mf{f}, \ev{\mf{f}}, M)
$$
is injective.
\end{theorem}

The proof is the same as in \cite{LNZ-2011} since $\mf{p}_{\pm 1}$ is an ideal of $\mf{p}$
and by use of Proposition \ref{P: b proj implies f proj} and
Lemma \ref{L: f zero implies p zero}.  This powerful result will be used in the form of
the following corollary.

\begin{corollary} \label{C: f detects p proj}
Let $M \in \F{p}$ such that $M|_{\mf{f}}$ is projective in $\F{f}$.  Then $M$ is projective
in $\F{p}$.
\end{corollary}
\begin{proof}
Let $S$ be a simple module in $\F{p}$.  Then 
$$
\H^1(\mf{p}, \ev{\mf{p}}, M \otimes S^*) \hookrightarrow \H^1(\mf{f}, \ev{\mf{f}}, M \otimes S^*) \cong \Ext^1_{\F{f}}(S, M) = 0
$$
since $M|_{\mf{f}}$ is projective in $\F{f}$.  Thus $\Ext^1_{\F{p}}(S, M) = 0$ as
well and $M$ is projective in $\F{p}$.
\end{proof}

\section{Restriction from $T(\mf{g})$ to $T(\mf{p})$} \label{S: res from Tg to Tp}
Let $\mf{g} = \mf{gl}(m|n)$ and $\mf{p} \subseteq \mf{g}$ be the
distinguished parabolic.
Since restriction from $T(\mf{g})$ to $T(\mf{p})$ is well defined
(Corollary \ref{C: T(g) to T(p)}),
properties of this map can be exploited to relate a classification of
one to the other.  An important step in understanding the
relationship between these two groups
is an induction functor from $\mf{p}$ to $\mf{g}$.

In \cite[Section 3]{GS-2010},
the geometric induction functor $\Gamma_0$ is defined.  The functor
$\Gamma_0$ is from $\mf{p}$-modules to $\mf{g}$-modules and
will be denoted $\Ind_\mf{p}^{\mf{g}}$ since the geometric structure
will not be emphasized in this paper.  This functor is of particular
interest because it will allow us to show that restriction map
$$\op{res}^{T(\mf{g})}_{T(\mf{p})}: T(\mf{g}) \rightarrow T(\mf{p})$$ given by $M \mapsto M|_{\mf{p}}$
is injective by
checking that $\ker \left( \op{res}^{T(\mf{g})}_{T(\mf{p})} \right)=
\{ k_{ev} \}$ since $k_{ev}$ is the identity in $T(\mf{g})$.

The first step in the proof is to show that
$\Ind_\mf{p}^{\mf{g}} k_{ev} = k_{ev}$.
This is done by considering
\cite[Lemma 3]{GS-2010} and its proof.  In particular,
the authors observe that if $L_{\mu}$ (respectively $L_{\mu}(\mf{a})$)
is the simple $\mf{g}$-module (respectively $\mf{a}$-module) with
highest weight $\mu$, then if $L_{\mu}$ occurs in
$\Ind_\mf{p}^{\mf{g}} k_\lambda$,
then $L_{\mu}(\ev{\mf{g}})^*$ occurs in
$H^0(G_0/P_0, \mc{L}^*_\lambda(\mf{p}) \otimes S^{\bullet} (\mf{g}/(\ev{\mf{g}} \oplus \od{\mf{p}}))^*)$.

The case when $k_\lambda$ is the trivial module $k_{ev}$ is of particular interest
as noted above.  Thus we consider
$H^0(G_0/P_0, S^{\bullet} (\mf{g}/(\ev{\mf{g}} \oplus \od{\mf{p}}))^*)$,
and more specifically, the dominant weights in
$S^{\bullet} (\mf{g}/(\ev{\mf{g}} \oplus \od{\mf{p}}))^*$.
In order for such a
weight to be dominant, it must have positive inner product with
$\varepsilon_1 - \varepsilon_2 , \varepsilon_2 - \varepsilon_3, \dots ,
\varepsilon_{m-1} - \varepsilon_m$ and
$\delta_1 - \delta_2, \delta_2 - \delta_3, \dots, \delta_{n-1} - \delta_n$.
The weights of $S^{\bullet} (\mf{g}/(\ev{\mf{g}} \oplus \od{\mf{p}}))^*$
are positive linear combinations of the weights of the form
$\varepsilon_i - \delta_j$ and
$\delta_i - \varepsilon_j$ where $i > j$.
\begin{proposition} \label{P: glnn no dominant}
Let $\mf{p} \subseteq \mf{g} = \mf{gl}(n|n)$.
No weight of $S^{\bullet} (\mf{g}/(\ev{\mf{g}} \oplus \od{\mf{p}}))^*$ is
dominant.
\end{proposition}
\begin{proof}
This will be proven by induction on $n$.  The first case is trivial
since when $n=1$, $\od{\mf{p}} = \od{\mf{g}}$ and
$\mf{g} = \supalg{\mf{g}}$.

The first nontrivial base case is when $n = 2$.  If the weights
$\varepsilon_2 - \delta_1$ and $\delta_2 - \varepsilon_1$ are represented
as $(0,1|-1,0)$ and $(-1,0|0,1)$ respectively, then a positive linear
combination of such weights $r(\varepsilon_2 - \delta_1) +
s(\delta_2 - \varepsilon_1)$ is represented as $(-s,r|-r,s)$.  We compute
\begin{gather*}
\langle (1,-1|0,0) ,(-s,r|-r,s) \rangle = -s -r \\
\langle (0,0|1,-1) , (-s,r|-r,s) \rangle = -s -r
\end{gather*}
and so any nonzero weight has negative inner product and thus, is not
dominant.

Now let $n > 2$.  In order for a positive linear combination of weights
to be dominant, there is a set of conditions which must be satisfied.
Let $\lambda$ be an arbitrary weight and let
$a_{i,j}$ be the coefficient for the weight
$\varepsilon_i - \delta_j$ and $b_{k,l}$ be the coefficient of
the weight $\delta_k - \varepsilon_l$, where $i > j$ and $k > l$.
Then
$$
\lambda = \left( \sum_{i > j} a_{i,j}(\varepsilon_i - \delta_j) \right) 
+ \left( \sum_{k > l} b_{k,l}(\delta_k - \varepsilon_l) \right)
$$
or if we denote $\alpha_{i,j} = a_{i,j}(\varepsilon_i - \delta_j)$
and $\beta_{k,l} = b_{k,l}(\delta_k - \varepsilon_l)$, then
$\lambda = \sum_{i > j} (\alpha_{i,j} + \beta_{i,j})$.

Note that
\begin{align*}
\langle \varepsilon_s - \varepsilon_{s+1}, \alpha_{i,j} \rangle
& = \delta_{s, i}a_{i,j} - \delta_{s+1, i}a_{i,j} \\
\langle \varepsilon_s - \varepsilon_{s+1}, \beta_{i,j} \rangle
& = -\delta_{s, j}b_{i,j} + \delta_{s+1,j}b_{i,j} \\
\langle \delta_s - \delta_{s+1}, \alpha_{i,j} \rangle
&= -\delta_{s, j}a_{i,j} + \delta_{s+1,j}a_{i,j} \\
\langle \delta_s - \delta_{s+1}, \beta_{i,j} \rangle
&= \delta_{s, j}b_{i,j} - \delta_{s+1,j}b_{i,j}
\end{align*}
where $\delta_{s,t}$ is the Kronecker delta.  We note that the conditions
$\langle \varepsilon_s - \varepsilon_{s+1}, \lambda \rangle \geq 0$ and
$\langle \delta_s - \delta_{s+1}, \lambda \rangle \geq 0$ for each
$s = 1, \dots, n-1$ give $2(n-1)$ inequalities which the coefficients
$a_{i,j}$ and $b_{i,j}$ must satisfy.

The important step in this proof is to add all the given inequalities
together to produce one inequality,
\begin{gather*}
\sum_{s=1}^{n-1} \left(
\langle \varepsilon_s - \varepsilon_{s+1}, \lambda \rangle +
\langle \delta_s - \delta_{s+1}, \lambda \rangle
\right) = \\
\sum_{s=1}^{n-1} \sum_{i > j} \left(
\langle \varepsilon_s - \varepsilon_{s+1},\alpha_{i,j} + \beta_{i,j} \rangle
+ \langle \delta_s - \delta_{s+1},\alpha_{i,j} + \beta_{i,j} \rangle
\right) \geq 0.
\end{gather*}
Next, observe
that each $a_{i,j}$ and $b_{i,j}$ appears exactly twice as a negative
term in the inequality.  Furthermore, each term $a_{k,l}$ and
$b_{k,l}$ with $1 < k,l < n$ appears twice as a positive term and
$a_{i,1}$, $a_{n,j}$, $b_{i,1}$, and $b_{n,j}$ appear at most once
as a positive term (with $a_{1,n}$ and $b_{1,n}$ being the terms which
do not appear at all).
Rearranging the inequality then yields
$$
0 \geq \sum_{s = 1}^n (a_{s,1} + a_{n,s} +  b_{s,1} + b_{n,s} )
$$
and so each coefficient of this form is forced to be zero in order
for a weight to be dominant.  However, by induction, we have now reduced
to a weight whose nonzero coefficients come from a lower diagonal
$(n-2) \times (n-2)$  matrix which has no dominant weights
by the inductive hypothesis.  Thus,
the claim is proven.
\end{proof}

In order to handle the general case of classifying $T(\mf{gl}(m|n))$
when $m \neq n$,
a slight modification to the previous argument must be made.  Although
the following proof suffices in general, the previous special
case is included
as it is helpful in clarifying this argument.

Define $r = \min(m,n)$ and $s = |m - n|$.
When $m \neq n$, the parabolic 
subalgebra $\mf{p}$ is now a bit different.
As detailed in Section \ref{SS: dist para}
The even component still consists of upper triangular matrices
(now of different sizes) but
the odd component is structurally different.  Since $\mf{g}_1$ and
$\mf{g}_{-1}$ are no longer square (they are $m \times n$ and $n \times m$),
the entries above the odd diagonal are no longer symmetric.  In particular,
$\dim(\mf{p}_1) \neq \dim (\mf{p}_{-1})$.  There is a subalgebra of
$\mf{p}$ isomorphic to the distinguished parabolic of $\mf{gl}(r|r)$ and
the previous argument can be applied to this subalgebra with only
a slight modification.

\begin{proposition}
Let $\mf{p} \subseteq \mf{g} = \mf{gl}(m|n)$.
No weight of $S^{\bullet} (\mf{g}/(\ev{\mf{g}} \oplus \od{\mf{p}}))^*$ is
dominant.
\end{proposition}

As before,
let $\lambda$ be an arbitrary weight and let
$a_{i,j}$ be the coefficient for the weight
$\varepsilon_i - \delta_j$ and $b_{k,l}$ be the coefficient of
the weight $\delta_k - \varepsilon_l$, where $i > j$ and $k > l$ and so
$$
\lambda = \left( \sum_{i > j} a_{i,j}(\varepsilon_i - \delta_j) \right) 
+ \left( \sum_{k > l} b_{k,l}(\delta_k - \varepsilon_l) \right).
$$
Before, we considered all the conditions of dominance all at once and
this was sufficient for the previous case.  Now, the conditions
will be considered in a particular order to achieve the result.

As noted, there is a canonical subalgebra $\mf{gl}(r|r) \subseteq
\mf{gl}(m|n)$ which we will denote as $\mf{g}_r$.   Furthermore,
$\mf{g}_r$ contains a distinguished parabolic subgroup as well which
will be denoted $\mf{p}_r$.  These subalgebras will provide
a useful reduction in this proof.

There are two main steps in the proof.  The first is to use an
induction argument to
eliminate the possibility of dominant roots in the portion of
$\mf{g}/(\ev{\mf{g}} \oplus \od{\mf{p}})$ isomorphic to
$\mf{g}_r/(\ev{(\mf{g}_r)} \oplus \od{(\mf{p}_r)})$ and so a variation on
the previous argument used here.  We proceed by induction on $r$.

If $r = 1$, then this first reduction step is trivial since $\od{(\mf{p}_r)}
= \od{(\mf{g}_r)}$ and so
$\mf{g}/(\ev{\mf{g}} \oplus \od{\mf{p}})$ is isomorphic to
$\mf{g}_1 / (\mf{g}_r)_1$ if $m > n$ and  $\mf{g}_{-1} / (\mf{g}_r)_{-1}$ if
$m < n$ which is the desired reduction.

If $r > 1$, then consider the conditions imposed by
$\langle \varepsilon_t - \varepsilon_{t+1}, \lambda \rangle \geq 0$ if
$m > n$, and consider
$\langle \delta_t - \delta_{t+1}, \lambda \rangle \geq 0$ if $m < n$,
where $1 \leq t \leq r-1$ in both cases.
For brevity, only the case where $m > n$ will be discussed,
as the proof for $m < n$ is very similar.

If $\lambda$ is an arbitrary weight as above, the inner products
$\langle \varepsilon_t - \varepsilon_{t+1}, \lambda \rangle$
for $1 \leq t \leq r-1$
have nontrivial interaction only with the part of the
weights which lies
in the subalgebra $\mf{g}_r/(\ev{(\mf{g}_r)} \oplus \od{(\mf{p}_r)})$.
Now, nearly the same technique as in Proposition \ref{P: glnn no dominant}
can be applied.

As before,
\begin{align*}
\langle \varepsilon_t - \varepsilon_{t+1}, a_{i,j}(\varepsilon_i - \delta_j) \rangle
& = \delta_{t, i}a_{i,j} - \delta_{t+1, i}a_{i,j} \\
\langle \varepsilon_t - \varepsilon_{t+1}, b_{i,j}(\delta_i - \varepsilon_j) \rangle
& = -\delta_{t, j}b_{i,j} + \delta_{t+1,j}b_{i,j} 
\end{align*}
and the inequalities given by $\langle \varepsilon_t - \varepsilon_{t+1}, \lambda \rangle \geq 0$ are again summed to yield
\begin{gather*}
\sum_{t=1}^{r-1} 
\langle \varepsilon_t - \varepsilon_{t+1}, \lambda \rangle  = 
\sum_{t=1}^{r-1} \sum_{i > j} 
\langle \varepsilon_t - \varepsilon_{t+1},\alpha_{i,j} + \beta_{i,j} \rangle
 \geq 0.
\end{gather*}
Since only half of the inequalities have been used in this case
(all the ones involving deltas have been left out),
each coefficient appears exactly once as a negative term and
the  $a_{i,j}$ and $b_{k,l}$ with $i < r$ and $l > 1$ appear
once as positive terms.  Rearranging the inequality gives
$$
0 \geq \sum_{t = 1}^n ( a_{r,t} +  b_{t,1} )
$$
which reduces to a case isomorphic to showing that there are no
dominant weights in $\mf{g}_{r-1}/(\ev{(\mf{g}_{r-1})} \oplus \od{(\mf{p}_{r-1})})$ which contains no dominant weights by the
inductive hypothesis which completes the first step.

Now, if a weight $\lambda$ is dominant, it must be a weight for
$\mf{g}_1 / (\mf{g}_r)_1$ since
all the other coefficients of $\lambda$ have been show to be 0.
This step is significantly easier since now
$$
\lambda = \sum_{\substack{  \ r < i  \leq m\\  m+1 < j \leq m+n}} a_{i,j}(\varepsilon_i - \delta_j)
$$
and the condition
$\langle \varepsilon_r - \varepsilon_{r+1}, \lambda \rangle \geq 0$
implies that
$$ 0 \geq \sum_{m+1 < j \leq m+n} a_{r+1, j}
$$
and so each of these coefficients is 0.  This process is repeated
stepwise for the conditions
$\langle \varepsilon_t - \varepsilon_{t+1}, \lambda \rangle \geq 0$
for $ r < t < m - 1$ which shows
that $a_{t, j} = 0$ for all  $m+1 < j \leq m+n$,
and finally, that $\lambda = 0$, which is not a dominant weight.  Thus,
no weight of $S^{\bullet} (\mf{g}/(\ev{\mf{g}} \oplus \od{\mf{p}}))^*$ is
dominant and the proof is complete.

\begin{corollary}
Let $\mf{p} \subseteq \mf{g} = \mf{gl}(m|n)$, then $
\Ind_{\mf{p}}^{\mf{g}} k \cong k$.
\end{corollary}
\begin{proof}
Since $S^{\bullet} (\mf{g}/(\ev{\mf{g}} \oplus \od{\mf{p}}))^*$ has
no dominant weights by the previous lemma,
$$
H^0(G_0/P_0, S^{\bullet} (\mf{g}/(\ev{\mf{g}} \oplus \od{\mf{p}}))^*)
\cong k.
$$
Furthermore, note that the induction functor does not change the
parity of the module, so the degree (either even or odd) is fixed
and the result is proven.
\end{proof}


\begin{corollary} \label{C: Res is inj}
The restriction map
$$\op{res}^{T(\mf{g})}_{T(\mf{p})}: T(\mf{g}) \rightarrow T(\mf{p})$$ given by $M \mapsto M|_{\mf{p}}$
is injective.
\end{corollary}
\begin{proof}
By Corollary \ref{C: T(g) to T(p)}, it is sufficient to check that $\ker \left( \op{res}^{T(\mf{g})}_{T(\mf{p})} \right)=
\{ k_{ev} \}$.

Let $M \in T(\mf{g})$ be an indecomposable endotrivial
module in $\F{g}$ such that $M|_{\mf{p}} \cong k_{ev} \oplus P$.  Then
\begin{equation*}
\Ind_{\mf{p}}^{\mf{g}} M|_{\mf{p}} \cong
	\Ind_{\mf{p}}^{\mf{g}} (k_{ev} \oplus P) \cong
	\Ind_{\mf{p}}^{\mf{g}} k_{ev} \oplus \Ind_{\mf{p}}^{\mf{g}} P
	\cong k_{ev} \oplus \Ind_{\mf{p}}^{\mf{g}} P.
\end{equation*}
However, since $M$ is already a $\mf{g}$-module, by the
tensor identity given in \cite[Lemma 1]{GS-2010}, 
\begin{equation*}
\Ind_{\mf{p}}^{\mf{g}} M|_{\mf{p}} \cong
\Ind_{\mf{p}}^{\mf{g}} (M|_{\mf{p}} \otimes k_{ev}) \cong
M \otimes \Ind_{\mf{p}}^{\mf{g}}  k_{ev} \cong
M \otimes k_{ev} \cong M
\end{equation*}
and so we have that, as $\mf{g}$-modules,
$M \cong k_{ev} \oplus \Ind_{\mf{p}}^{\mf{g}} P$.
Since $M$ is indecomposable, $M \cong k_{ev}$ and 
thus the map $\op{res}^{T(\mf{g})}_{T(\mf{p})}$ is injective.
\end{proof}


\section{Classification of $T(\mf{gl}(m|n))$} \label{S: classification of T(g)}

\subsection{Classification of $T(\mf{p})$}

Using Corollary \ref{C: Res is inj}, we will use the classification of $T(\mf{p})$
to yield a classification of $T(\mf{g})$.  The computation of
$T(\mf{p})$ is derived by considering the kernel of
the restriction map given in Corollary \ref{C: T(p) to	 T(f)}.

For the remainder of the section let $\mf{g} = \mf{gl}(m|n)$ with
maximal torus $\ev{\mf{t}}$ and let $\mf{p}$
be the distinguished parabolic subalgebra of $\mf{g}$ such that
$\mf{f} \subseteq \mf{p} \subseteq \mf{g}$ and
$\ev{\mf{f}} = \mf{t_f} \subseteq \mf{t_p} = \ev{\mf{t}}$.
Let
$r = \min(m,n)$ and $s = |m - n|$.

\begin{theorem} \label{T: classification of T(p)}
There are isomorphisms of groups
\begin{enumerate}
\item $T(\mf{f}) \cong k^s \times \Z \times \Z_2$;
\item $T(\mf{p}) \cong k^{r+s} \times \Z \times \Z_2$.
\end{enumerate}
\end{theorem}
\begin{proof}
First, we recall some details about $\mf{f}$.
It is important to note that in this more general setting, 
the detecting subalgebra $\mf{f}$ may not be isomorphic to
$\mf{sl}(1|1) \times \dots \times \mf{sl}(1|1)$ as is assumed in \cite{Talian-2013}
because there are more entries on the even diagonal.  
If $m \neq n$, then define
$\mf{f} \cong \mf{f}_r \oplus \mf{t}_s$ where $\mf{f}_r$ is a direct sum of
$r$ copies of $\mf{sl}(1|1)$, the detecting subalgebra of $\mf{gl}(r|r) \subseteq \mf{gl}(m|n)$,
and $\mf{t}_s$ is an $s$ dimensional torus generated by the remaining diagonal
entries whose span does not intersect the torus of $\mf{f}_r$.  Thus,
the classification of $\mf{f}$-endotrivial modules in \cite{Talian-2013} must be modified
slightly.

First let $L$ be an indecomposable endotrivial module in $\frel{f}$.
The classification given in \cite[Theorem 6.2]{Talian-2013}
indicates that
$L|_{\mf{f}_r} \cong \Omega^i_{\mf{f}_r}(k|_{\mf{f}_r}) \oplus P'$ where
$P'$ is some projective $\mf{f}_r$-module, $k|_{\mf{f}_r}$ is either
$k_{ev}$ or $k_{od}$, and $n \in \Z$.  Since $\mf{f}$ is just $\mf{f}_r$
with an enlarged torus, the structure of $\mf{f}$-modules is not
fundamentally different, there are just more weights in $\mf{f}$.
So an endotrivial $\mf{f}$-module is an endotrivial $\mf{f}_r$-module
by restriction, which has an action of an $s$ dimensional torus
that may act by any weight
because $[\mf{f}, \mf{f}] \cap \mf{t}_s = 0$ (see 
\cite[Proposition 7.4]{Talian-2013} for more on this topic).  We may
then conclude that
$L \cong \Omega^i_{\mf{f}}(k_\lambda|_{\mf{f}}) \oplus P'$
where $P'$ is a projective $\mf{f}$-module.
and $\lambda \in X(\ev{(\mf{f}_r)} \oplus \mf{t}_s)$ of the form
$(0, \dots, 0, \lambda_{2r+1}, \dots, \lambda{_{n+m}})$, i.e.,
$k_\lambda|_{\mf{f}_r} \cong k$ the trivial module concentrated
in even or odd degree.  Thus, $T(\mf{f}) \cong k^s \times \Z \times \Z_2$
in this setting.

Given this observation, consider the map $\op{res}^{T(\mf{p})}_{T(\mf{f})}: T(\mf{p}) \rightarrow T(\mf{f})$ and in particular, its kernel.
Let $M$ be an indecomposable endotrivial module in $\F{p}$ such that
$M|_{\mf{f}} \cong k_{ev} \oplus P$.  Now, consider the weight space decomposition of
$M|_{\mf{f}}$ relative to the weights of $\mf{t_f}$.  Then
\begin{equation} \label{E:f decomp}
M \cong \bigoplus_{\lambda \in X(\mf{t_f})} M_\lambda
\end{equation}
is not only a direct sum over $\mf{t_f}$, but as a module in $\F{f}$ since
$[ \mf{t_f}, \od{\mf{f}} ] = 0$.  Thus, by comparing the
direct summands of $M|_{\mf{f}} \cong k_{ev} \oplus P$
with those in Equation \ref{E:f decomp}, we note that the
only non-projective summand $k_{ev}$ occurs in the block $M_0$ and
thus $M_\lambda$
is projective in $\F{f}$ when $\lambda \neq 0$.

Define $\mf{u} = \mf{p} / \mf{t_p}$ and consider the action of $\mf{u}$ on $M$ relative
to the direct sum decomposition of $M$ over $\mf{f}$.  For $u \in \mf{u}$,
$u.M_\lambda \subseteq M_\mu$ implies that $\mu > \lambda$
in the dominance ordering of
$X(\mf{t_f})$ defined in Equation \ref{E:Dominance ordering}.
Define 
$$
\hat{M} = \bigoplus_{\substack{\lambda \in X(\mf{t_f}) \\ \lambda \nleq 0 }} M_\lambda
$$
then $\hat{M}$ is a $\mf{p}$-submodule of $M$ by construction.  Observe that 
if $\hat{M} \neq 0$, then $\hat{M}$ is a module in $\F{p}$ such that $\hat{M}|_{\mf{f}}$
is projective in $\F{f}$ and so by Corollary \ref{C: f detects p proj}, $\hat{M}$ is
projective in $\F{p}$.
Since $\F{p}$ is self injective, projective modules are
also injective and
so this gives a splitting
$M \cong \hat{M} \oplus \hat{M}^c$ which is a contradiction
since $M$ was assumed to be indecomposable and so $\hat{M} = 0$.

Then $M$ decomposes as
$$
M \cong \bigoplus_{\substack{\lambda \in X(\mf{t_f}) \\ \lambda \leq 0 }} M_\lambda
$$
in $\F{f}$.  Note that $M_0$ is a $\mf{p}$-submodule of $M$ and define $\tilde{M} = M / M_0$.
Again, $\tilde{M}$ is a $\mf{p}$-module which is projective when restricted to $\mf{f}$.
Thus
$M \cong \tilde{M} \oplus \tilde{M}^c$ and we conclude $\tilde{M} = 0$ as well.

So it must be that $M = M_0$ and so $\mf{u}.M = 0$.  Since $M$ is in $\F{p}$, which by definition
has a weight space decomposition relative to $\mf{t_p}$, and since
$\mf{p} \cong \mf{t_p} \oplus \mf{u}$, the decomposition
$M|_{\mf{f}} \cong k_{ev} \oplus P$ is also a decomposition over $\mf{p}$.  Thus,
$M|_{\mf{f}} \cong k_{ev}$ and so $M$ is a one dimensional module whose weights over
$\mf{t_p}$ collapse to the trivial weight when restricted to $\mf{t_f}$.

So $M \cong k_\lambda$ where $\lambda = (\lambda_1, \dots, \lambda_r, \lambda_{r+1},
\dots, \lambda_{2r}, 0, \dots, 0)$ and
$\lambda_i = - \lambda_{r + i}$ for $i = 1, \dots, r$ and we see that
$\ker (\op{res}^{T(\mf{p})}_{T(\mf{f})}) \cong k^r$.

Recalling the notation $\mf{f} \cong \mf{f}_r \oplus \mf{t}_s$, we now have shown that
$T(\mf{p}) \cong k^{r+s} \times \Z \times \Z_2$, generated by
$\Omega^1_{\mf{p}}(k_{ev})$, $k_\lambda$ such that $k_\lambda|_{\mf{f_r}} \cong k_{ev}$ and the parity change functor.
\end{proof}

\subsection{Classification of $T(\mf{gl}(m|n))$}
The final step in the classification results from making a few observations
about the
results of Corollary \ref{C: Res is inj} and Theorem \ref{T: classification of T(p)}.


\begin{theorem} \label{T: classification of T(g)}
There is an isomorphism of groups
$$T(\mf{g}) \cong k \times \Z \times \Z_2.$$
\end{theorem}
\begin{proof}
It was shown that the map $\op{res}^{T(\mf{g})}_{T(\mf{p})}: T(\mf{g}) \rightarrow T(\mf{p})$ is injective, but now that $T(\mf{p})$ has been classified, the image
of the injection can be computed directly.

By
Lemma \ref{L: restriction commutes},
$$
\Omega^i_{\mf{g}}(M)|_{\mf{p}} \cong \Omega^i_{\mf{p}}(M|_{\mf{p}})
\oplus P
$$
where $P$ is a projective module in $\F{p}$ and so in the respective stable module
categories (where $T(\mf{g})$ and $T(\mf{p})$ are defined), the syzygy operation commutes
with restriction.  Additionally, the parity change functor commutes as well and
so it is clear that $\Omega^1_{\mf{g}}(k_{ev})$ and parity change functor generate
a subgroup of $T(\mf{g})$ isomorphic to $\Z \times \Z_2$.

The last factor of $T(\mf{p})$ is $k^{r+s}$ which arises from the one dimensional
modules in $\F{p}$.  There are fewer one dimensional modules in $\mf{g}$ (except
when $m = n = 1$ in which case $\mf{g} = \mf{p}$) because
$[\mf{p}, \mf{p}] \cap \ev{\mf{t}} \subseteq [\mf{g}, \mf{g}] \cap \ev{\mf{t}}$.
In $\mf{g}$ there is only one parameter of one dimensional modules whose weights
relative to $\ev{\mf{t}}$ are $(\lambda_1, \dots, \lambda_m, \lambda_{m+1}, \dots,
\lambda_{m+n})$ where $\lambda_i = -\lambda_j$ for all
$1 \leq i \leq m < j \leq m+n$.

Noting that $\mf{p}$ and $\mf{g}$ share the same torus, $k_\lambda \in \F{g}$ restricts
to $k_\lambda \in \F{p}$, and we now have a complete description of the image of the
restriction map.  An arbitrary indecomposable endotrivial module in $\F{g}$
is of the form $\Omega^i_{\mf{g}}(k_\lambda)$ for $\lambda \in X(\ev{\mf{t}})$
such that $\lambda = (\lambda_1, \dots, \lambda_m, \lambda_{m+1}, \dots,
\lambda_{m+n})$ where $\lambda_i = -\lambda_j$ for
all
$1 \leq i \leq m < j \leq m+n$ and
$$T(\mf{g}) \cong k \times \Z \times \Z_2 \hookrightarrow T(\mf{p}) \cong k^{r+s} \times \Z \times \Z_2 $$
where the injection is given by restriction .  Thus, 
$T(\mf{g})$ is
generated by $\Omega^1_{\mf{g}}(k_{ev})$, $k_\lambda$
concentrated in even degree where $\lambda$ is
as above, and the parity
change functor.
\end{proof}

\section*{Acknowledgments}
The author would like to express sincere thanks to
Daniel Nakano for his suggestions about how to approach the questions dealt with in this
work.  Without his many helpful insights, this paper would not be possible.  The author would also like to thank
Jonathan Kujawa for providing assistance and
feedback on drafts on this paper
which improved the clarity and quality of the writing.
Finally, the author is grateful to Institut Mittag-Leffler
for their hospitality and support during the completion of this project.

\end{document}